\documentclass[12pt,reqno]{amsart}
\usepackage[a4paper,hscale=0.7,vscale=0.7,centering]{geometry}
\usepackage{amssymb}

\newcommand{\al}{\alpha}
\newcommand{\Om}{\Omega}

\newcommand{\del}{\delta}
\newcommand{\Del}{\Delta}
\newcommand{\ep}{\varepsilon}

\newcommand{\Lam}{\Lambda}
\newcommand{\Gam}{\Gamma}
\newcommand{\gam}{\gamma}
\newcommand{\ffi}{\varphi}
\newcommand{\sig}{\sigma}

\newcommand{\hac}{\mathcal{H}}
\newcommand{\hact}{\mathcal{H}_T}
\newcommand{\cinf}{\mathcal{C}^\infty}
\newcommand{\N}{\mathbb{N}}
\newcommand{\re}{\mathbb{R}}
\newcommand{\red}{\mathbb{R}^{d}}
\newcommand{\tf}{\mathcal{F}}

\newcommand{\cs}{\mathcal{S}}
\newcommand{\cl}{\mathcal{L}}
\newcommand{\R}{\mathbb{R}}
\newcommand{\D}{\mathbb{D}}

\newcommand{\E}{\mathbb{E}}
\renewcommand{\P}{\mathbb{P}}

\newcommand{\beq}{\begin{equation}}
\newcommand{\eeq}{\end{equation}}

\newtheorem{prop}{Proposition}[section]
\newtheorem{thm}[prop]{Theorem}

\newtheorem{lemma}[prop]{Lemma}

\theoremstyle{definition}
\newtheorem{hyp}[prop]{Hypothesis}
\theoremstyle{remark}

\newtheorem{example}[prop]{Example}

\numberwithin{equation}{section}

\begin{document}

\title[Density estimates]{Gaussian upper density estimates for spatially homogeneous SPDEs}

\author{Llu\'is Quer-Sardanyons}
\address{Departament de Matem\`atiques,
  Universitat Aut\`onoma de Barcelona, 08193 Bellaterra (Barcelona),
  Spain}
\email{quer@mat.uab.cat}
\urladdr{http://www.mat.uab.cat/~quer}

\begin{abstract}
We consider a general class of SPDEs in $\mathbb{R}^d$ driven by a Gaussian spatially homogeneous noise which is white in time. 
We provide sufficient conditions on the coefficients and the spectral measure associated to the noise ensuring that 
the density of the corresponding mild solution admits an upper estimate of Gaussian type.
The proof is based on the formula for the density arising from the integration-by-parts formula of the Malliavin calculus. 
Our result applies to the stochastic heat equation with any space dimension and the stochastic wave equation with
$d\in \{1,2,3\}$. In these particular cases, the condition on the spectral measure turns out to be optimal. 
\end{abstract}



\thanks{Supported by the grant MICINN-FEDER Ref. MTM2009-08869}

\date{\today}

\maketitle

\section{Introduction}

We are interested in establishing Gaussian type upper estimates for the density of the {\emph{mild}} solution of the following class of SPDEs:
\beq
\cl u(t,x) =  b(u(t,x)) + \sigma(u(t,x)) \dot{W} (t,x), \qquad (t,x)\in [0,T]\times \red,
\label{eq:21}
\eeq
where $T>0$ is some fixed time horizon and $\cl$ denotes a general
second order partial differential operator with constant coefficients, with appropriate initial
conditions. The coefficients $\sigma$ and $b$ are real-valued
functions and $\dot{W} (t,x)$ is the formal notation for a
Gaussian random perturbation which is white in time and has some spatially homogeneous correlation 
(see Section \ref{sec:noise} for a precise definition of this noise). 
The typical examples of operator $\cl$ to which our result applies are the heat operator 
for any spatial dimension $d\geq 1$ and the wave operator with $d\in\{1,2,3\}$. 

If $\cl$ is first order in time, such as the heat operator
$\cl=\frac{\partial}{\partial t}-\Delta$, where $\Delta$ denotes the Laplacian operator on $\mathbb{R}^d$, then we impose initial
conditions of the form 
\beq 
u(0,x)=u_0(x) \qquad x\in \red,
\label{2.1.1} 
\eeq 
for some Borel function $u_0:\red \rightarrow \re$. If $\cl$ is second order in time, such as the wave operator
$\cl=\frac{\partial^2}{\partial t^2}-\Delta$, then we have to impose two initial conditions: 
\beq 
u(0,x)=u_0(x), \quad \frac{\partial u}{\partial t}(0,x)=v_0(x),\qquad x\in \red, 
\label{2.1.2} 
\eeq 
for some Borel functions $u_0,v_0:\red \rightarrow \re$.

The above class of SPDEs has been widely studied in the last two decades. 
Precisely, results on existence and uniqueness of solution in such a general setting have been established in
\cite{Peszat-Zabczyk-SPA1997,Dalang-EJP1999,Dalang-MuellerEJP2003,Dalang-Quer}, while the particular cases of heat and wave equations
have been studied using several frameworks in 
\cite{Walsh-LNM,Carmona-NualartPTRF1988,Dalang-FrangosAP1998,Millet-Sanz-AP1999,Karczewska-Zabczyk2000,Peszat-Zabczyk-PTRF2000,Peszat-JEE2002,Conus-Dalang}.

A fruitful line of research developed in some of the above-cited references has been to apply techniques of Mallliavin calculus 
in order to deduce some interesting properties of the probability law of the solution at any $(t,x)\in (0,T]\times \red$.
In fact, there is a whole bunch of results on existence and smoothness of the density for the stochastic heat and wave 
equations, for which we refer to 
\cite{Carmona-NualartPTRF1988,Pardoux-Zhang-JFA93,Bally-PardouxPOTA1998,Millet-Sanz-AP1999,Marquez-Mellouk-SarraSPA2001,Quer-Sanz-JFA,Quer-Sanz-Bernoulli,Sanz-book,Marinelli-Nualart-Quer}.
Moreover, in the paper \cite{Nualart-QuerPOTA}, existence and smoothness of density for the class of SPDEs (\ref{eq:21}) have been analyzed, 
unifying and improving some of the results cited so far. It is also worth mentioning that other kind of methods beyond Malliavin calculus 
can be used to prove the absolute continuity of the law of the solution in some particular SPDEs (see e.g. \cite{Fournier}).

Once the existence (and possibly smoothness) of the density of the solution to (\ref{eq:21}) is established, one usually gathers at some {\emph{nice}} 
estimates for this density, 
such as lower and upper Gaussian type bounds. Exploiting again techniques of Malliavin calculus, this problem has been recently addressed
by several authors. Precisely, as far as SPDEs with additive noise is concerned, using an explicit formula for the density proved in \cite{Nourdin-Viens}, 
the main result in \cite{Nualart-Quer-IDAQP} 
says the following (see Thm. 7 therein, and also \cite{Nualart-Quer-SPA} for related results). Let $\mu$ be spectral measure associated with the spatial correlation of $\dot W$, 
and consider the following assumption, which is necessary and sufficient for the existence and uniqueness of mild solution to (\ref{eq:21}) (see e.g. \cite{Dalang-EJP1999}).

\begin{hyp}\label{hyp:mu}
 Let $\Gam$ be the fundamental solution associated to the operator $\cl$. For all $t>0$, $\Gamma(t)$ defines a non-negative distribution with rapid decrease 
such that 
\[
\Phi(T):= \int_0^T\!\!\int_{\red} |\tf \Gam(t)(\xi)|^2 \, \mu(d\xi)dt < +\infty.
\]
Moreover, $\Gamma$ is a non-negative measure of the form $\Gamma(t,dx)dt$ such that, for all $T>0$,
\[
\sup_{0\leq t\leq T} \Gamma(t,\mathbb{R}^{d})< + \infty. 
\]
\end{hyp}

\bigskip

Then, under the above hypothesis, with vanishing initial data, $\sig\equiv 1$ and assuming that $b\in \mathcal{C}^1$ has a bounded derivative, 
\cite[Thm. 7]{Nualart-Quer-IDAQP} states that, for small enough $t$ and any $x\in \red$, the density $p_{t,x}$ of $u(t,x)$ satisfies,
for almost every $z\in \mathbb{R}$,
\beq
\frac{E\left| u(t,x)-m\right|}{ C_2 \Phi(t)} \exp\left\{ -\frac{(z-m)^2}{C_1 \Phi(t) }
\right\}\leq p_{t,x}(z) \leq \frac{E\left| u(t,x)-m\right|}{ C_1 \Phi(t)} \exp\left\{ -\frac{(z-m)^2}{C_2 \Phi(t)}\right\},
\label{eq:10}
\eeq
where $m=E(u(t,x))$, for some positive constants $C_1, C_2$. Note that the term $\Phi(t)$ is precisely the variance of the 
stochastic convolution in the mild form of (\ref{eq:21}) when $\sig\equiv 1$. As a consequence, this result applies to the stochastic
heat equation for any $d\geq 1$ and the stochastic wave equation in the case $d\in \{1,2,3\}$ provided that (see \cite[Sec. 3]{Dalang-EJP1999})
\beq
 \int_{\red}\frac{1}{1+|\xi|^2}\,\mu(d\xi) < +\infty.
\label{eq:101}
\eeq

On the other hand, in the multiplicative noise setting, such kind of density estimates, particularly the lower one, become more difficult to obtain
and Nourdin-Viens' density formula cannot be applied.
This has been already illustrated by Kohatsu-Higa in \cite{Kohatsu-PTRF2003} where, by means of Malliavin calculus techniques, 
a new method to obtain Gaussian lower bounds for 
general functionals of the Wiener sheet has been obtained. In the same paper, the author has applied this result to a 
stochastic heat equation in $[0,1]$ and driven by the space-time white noise.    
 
In order to deal with SPDEs beyond the one-dimensional setting, in \cite{Eulalia-Quer-SPA} Kohatsu-Higa's general method 
has been extended to the Gaussian space associated to the underlying spatially homogeneous noise $\dot W$. This allowed us 
to end up with the following density estimates for the stochastic heat equation in any space dimension $d\geq 1$. 
Assume that $b,\sig\in \cinf$ are bounded together with all their derivatives, $|\sig(z)|\geq c>0$ for all $z\in \re$, 
and for some $\eta\in (0,1)$, it holds
\beq
 \int_{\red}\frac{1}{(1+|\xi|^2)^\eta}\,\mu(d\xi) < +\infty.
\label{eq:100}
\eeq
Then, for all $(t,x)\in (0,T]\times \red$, the density $p_{t,x}$ of $u(t,x)$ verifies, for all $z\in \re$,
\begin{align}
& C_1 \Phi(t)^{-1/2} \exp \left\{ -\frac{|z-I_0(t,x)|^2}{C_2 \Phi(t)} \right\} \leq p_{t,x}(z) \nonumber \\
& \qquad \qquad \leq c_1 \Phi(t)^{-1/2} \exp \left\{ -\frac{(\vert z-I_0(t,x)\vert-c_3 T)^2}{c_2 \Phi(t)}\right\},
\label{eq:9}
\end{align}
where $I_0(t,x)=(\Gam(t) \ast u_0)(x)$, $u_0$ being the initial data. In the case $b\equiv 0$, the constant $c_3$ would vanish.
Note that here, in comparison to (\ref{eq:10}), the estimates are valid for any $T>0$.

In fact, let us point out that the upper bound in (\ref{eq:9}) is much easier to obtain than the lower one, and 
the former comes from the expression for the density popping up from the integration-by-parts formula in the Malliavin calculus framework.

Extending the lower estimate in (\ref{eq:9}) to the general class of SPDEs (\ref{eq:21}) seems to be an open problem, 
for the success in the application of the general strategy of \cite{Eulalia-Quer-SPA} is closely tied to 
the parabolic structure of the heat equation. However, a much more humble objective, which is the one we plan to gather in the 
present paper, is to tackle the upper bound. In fact, we are going to seek the minimal conditions on either  
the coefficients $b$ and $\sig$ and the spectral measure $\mu$ implying that the upper estimate in (\ref{eq:9}) remains valid for the general class of SPDEs (\ref{eq:21}). 
In particular, we will only need $b$ and $\sig$ to be of class $\mathcal{C}^2$ (and bounded with bounded derivatives) and, 
for the particular case of the heat (resp. wave) equation with any $d\geq 1$ (resp. $d\in \{1,2,3\}$), the condition on $\mu$ 
will be simply (\ref{eq:101}) rather that (\ref{eq:100}). More precisely, the main result of the paper is the following.
We use the notation $I_0(t,x)$ to denote the contribution of the initial data (see (\ref{2.3}) and (\ref{2.4}) for its explicit expression in the case of heat and
wave equations) and suppose that the forthcoming Hypothesis \ref{hyp:ic} is satisfied.

\begin{thm}\label{thm:main}
 Assume that Hypothesis \ref{hyp:mu} is satisfied, and that $b, \sig \in \mathcal{C}^2$ are bounded, have bounded derivatives, 
and $|\sig(z)|\geq c>0$ for all $z\in \re$. Moreover, suppose that, for some $\gam>0$, it holds 
\beq
 C\, \tau^\gam \leq \Phi(\tau)=\int_0^\tau \!\!\int_{\red} |\tf \Gam(s)(\xi)|^2 \, \mu(d\xi)ds, \qquad \tau\in (0,1].
\label{eq:102}
\eeq
Then, for any $(t,x)\in (0,T]\times \red$, the solution $u(t,x)$ of (\ref{eq:21}) has a density $p_{t,x}$ which is a continuous function and satisfies,
for all $z\in \re$,
\beq
 p_{t,x}(z) \leq c_1 \Phi(t)^{-1/2} \exp \left\{ -\frac{(\vert z-I_0(t,x)\vert-c_3 T)^2}{c_2 \Phi(t)}\right\},
\label{eq:gb}
\eeq
where the constant $c_3$ vanishes whenever $b$ does.
\end{thm}

We remark that, though the above bound does not look exactly Gaussian, it does in an 
{\emph{asymptotic}} point of view, namely whenever $T$ is small or $z$ is large. 
On the other hand, we note that, under (\ref{eq:101}), condition (\ref{eq:102}) is satisfied for the heat and wave equations with $\gam=1$ and $\gam=3$, respectively 
(see e.g. \cite[Lem. 3.1]{Marquez-Mellouk-SarraSPA2001} and \cite[App. A]{Quer-Sanz-JFA}). Similarly, one can also check that the above theorem 
applies to the stochastic damped wave equation with any space dimension (see Example 7 in \cite[Sec. 3]{Dalang-EJP1999}), where condition (\ref{eq:102}) is fulfilled with $\gam=3$. 

We also point out that our result is not applicable to the case $\sig(z)=z$ (this would be related, e.g., to the parabolic Anderson problem \cite{Carmona-Molchanov}). 
In fact, in such a case there are even very few results on absolute continuity of the law of solutions to SPDEs (see \cite{Pardoux-Zhang-JFA93}). 
Nevertheless, in the recent paper \cite{Hu-Nualart-Song}, the authors prove existence and smoothness of the density for a stochastic heat equation with 
a nonlinear multiplicative noise which is white in time and with some spatial correlation (much more regular than the one considered in the 
present paper), and with a non-degeneracy condition on the diffusion coefficient of the form $\sigma(u_0(x_0))\neq 0$ for some $x\in \red$.
Their proof is based on a Feynman-Kac formula for the solution of the underlying equation. This technique has also been applied in  
\cite{Hu-Nualart-Song-AP2011} to study the density for a stochastic heat equation with a linear multiplicative fractional Brownian sheet.

As mentioned before, the proof of Theorem \ref{thm:main} will be based on the expression for the density arising from 
the application of the integration-by-parts formula (see \cite[Prop. 2.1.1]{Nualart-llibre}). We point out that this is a well-known method
that has been used in other contexts (see e.g. \cite{Guerin-Meleard-Eulalia,Dalang-Khosh-Nualart-PTRF2009}).   
As far as the technical obstacles is concerned, the main two ingredients needed in the proof of Theorem \ref{thm:main} are the following:

\begin{itemize}
 \item[(i)] A suitable estimate, in terms of $\Phi(t)$, of the norm of the iterated Malliavin derivative in a small time interval (see Lemma \ref{lemma:0}
for details). This will be a consequence
of a kind of analogous result for the case of the stochastic heat equation (see \cite[Lem. 3.4]{Eulalia-Quer-SPA}) and a mollifying 
procedure thanks to an approximation of the identity which will let us smooth the fundamental solution $\Gam(t)$. 
\item[(ii)] A precise control of the negative moments of the norm of the Malliavin derivative of the solution, again in terms of $\Phi(t)$ (see Proposition 
\ref{prop:matrix}). For this, 
we will adapt the proof of \cite[Thm. 6.2]{Nualart-QuerPOTA} to our setting, where the latter allowed the authors of that paper to establish 
that the underlying density is a smooth function (under much more regularity on the coefficients though).
\end{itemize}

The content of the paper is organized as follows. In Section \ref{sec:prel}, we rigorously describe the Gaussian spatially homogeneous noise 
$\dot W$ considered in equation (\ref{eq:21}), we introduce the corresponding Gaussian setting associated to it, together with the main notations of 
the Malliavin calculus machinery. Section \ref{sec:spde} will be devoted to recall the definition of mild solution to our SPDE (\ref{eq:21}) and 
summarize the main results on existence and uniqueness of solution, Malliavin differentiability and existence and smoothness of the density. 
The steps (i) and (ii) detailed above will be tackled in Section \ref{sec:aux}. Finally, we will prove Theorem \ref{thm:main} in Section \ref{sec:upper-bound}.    

\medskip

With a slight (but harmless) abuse of notation, as already done in this Introduction, the notation $|\cdot|$ shall 
denote either the modulus and norm in $\red$. 
Unless otherwise stated, any constant $c$ or $C$ appearing in our computations below
is understood as a generic constant which might change from line to line without further mention.

\section{Preliminaries}
\label{sec:prel}

\subsection{Spatially homogeneous noise}
\label{sec:noise}

Let us explicitly describe here our spatially homogeneous noise (see e.g. \cite{Dalang-EJP1999}). Precisely, on a complete probability space $(\Omega, \mathcal{F}, \P)$,
this is given by a family $W=\{W(\varphi), \varphi \in \mathcal{C}^{\infty}_0(\re_+ \times \R^{d})\}$ of zero mean Gaussian 
random variables, where $\mathcal{C}_0^\infty(\re_+ \times \R^{d})$ denotes the space of smooth functions with compact support, with the 
following covariance structure:
\begin{equation}
\E \big( W(\varphi) W(\psi) \big) = \int_0^\infty \!\! \int_{\R^d} \left( \varphi(t,\star)*\tilde \psi(t,\star)\right)(x) \, \Lam(dx) dt.
\label{eq:00}
\end{equation}
In this expression, $\Lam$ denotes a non-negative and non-negative definite tempered measure on $\R^d$, 
$*$ stands for the convolution product, the symbol $\star$ denotes the spatial variable and $\tilde \psi(t,x):=\psi(t,-x)$.

In the above setting, a well-known result of harmonic analysis (see \cite[Chap. VII, Th\'eor\`eme XVII]{Schwartz}) 
implies that $\Lam$ has to be the Fourier transform of a non-negative tempered measure $\mu$ on $\R^d$, 
where the latter is usually called the {\emph{spectral measure}} of the noise $W$. We recall that, 
in particular, for some integer $m \geq 1$ it holds
\[
\int_{\R^d} \frac{1}{(1+|\xi|^2)^m} \, \mu(d\xi) < +\infty
\] 
and, by definition of the Fourier transform in the space $\cs'(\red)$ of tempered distributions, 
$\Lam =\tf \mu$ means that, for all $\phi$ belonging to the space $\cs(\red)$ 
of rapidly decreasing $\cinf$ functions,
\[
\int_{\red} \phi(x)\Lam(dx)=\int_{\red} \tf \phi(\xi) \mu(d\xi).
\]
Therefore, we have 
\[
\E \big( W(\varphi)^2 \big) = \int_0^\infty \!\! \int_{\R^d} |\mathcal{F} \varphi(t)(\xi)|^2 \mu(d\xi) dt.
\]

A typical example of space correlation is given by $\Lam(dx)=f(x)dx$, where $f$ is a non-negative function which is 
assumed to be integrable around the origin. In this case, the covariance functional (\ref{eq:00}) reads
\[
\int_0^\infty \!\! \int_{\red} \!\! \int_{\red} \ffi(t,x) f(x-y) \psi(t,y) \, dy dx dt.
\]
The space-time white noise would correspond to the case where $f$ is the Dirac delta at the origin.

\subsection{Gaussian setting and Malliavin calculus}
\label{sec:malliavin}

We are going to describe the Gaussian framework which can be naturally associated to our noise $W$ 
and introduce the notations involved in the Malliavin calculus techniques.  

To start with, let us denote by $\mathcal{H}$ be the completion of the Schwartz space $\mathcal{S}(\R^{d})$ endowed with the semi-inner product
\[
\langle \phi_1, \phi_2 \rangle_{\mathcal{H}}:= \int_{\R^d}  (\phi_1 * \tilde{\phi_2})(x) \, \Lam(dx) =
\int_{\R^d} \mathcal{F} \phi_1(\xi) \overline{\mathcal{F} \phi_2 (\xi)} \, \mu(d\xi), \quad \phi_1, \phi_2 \in \mathcal{S}(\R^{d}).
\]
As proved in \cite[Example 6]{Dalang-EJP1999}, we remind that the Hilbert space $\mathcal{H}$ may contain distributions. 

Let $T>0$ be a fixed real number and define  $\mathcal{H}_T:=L^2([0,T]; \mathcal{H})$.
Using an approximation argument, our noise $W$ can be extended to a family of mean zero Gaussian 
random variables indexed by $\mathcal{H}_T$ (see e.g. \cite[Lemma 2.4]{Dalang-Quer}). With an innocuous abuse of notation, this family will be still denoted by 
$W=\{W(g), g \in \mathcal{H}_T\}$. Moreover, it holds
$\E\big( W(g_1)W(g_2) \big)=\langle g_1,g_2 \rangle_{\mathcal{H}_T}$, 
for all $g_1,g_2 \in \mathcal{H}_T$. Thus, this family defines an {\emph{isonormal Gaussian process}} on the Hilbert space $\hact$ and
we shall use the differential 
Malliavin calculus based on it (see e.g. \cite{Nualart-llibre,Sanz-book}). 

As usual, we denote the Malliavin derivative operator by
$D$. Recall that it is a closed and unbounded operator defined in $L^2(\Om)$ and taking values in  $L^2(\Omega;\mathcal{H}_T)$, 
whose domain is denoted by $\mathbb{D}^{1,2}$. More general, for any integer $m \geq 1$ and any $p \geq 2$, the domain of the iterated
Malliavin derivative $D^m$ in $L^p(\Om)$ will be denoted by $\mathbb{D}^{m,p}$,  
where we remind that $D^m$ takes values in $L^p(\Omega; \mathcal{H}_T^{\otimes m})$. We also set
$\mathbb{D}^{\infty}= \cap_{p \geq 1} \cap_{m \in \N}
\mathbb{D}^{m,p}$. The space $\mathbb{D}^{m,p}$ can also be seen as the completion of the set of {\emph{smooth functionals}} 
with respect to the semi-norm
\[
\|F\|_{m,p}:=\Bigg\{\E \left( \vert F \vert^p\right)+\sum_{j=1}^m \E\Big( \Vert D^j F \Vert^p_{\hact^{\otimes j}}\Big)\Bigg\}^{\frac 1p}.
\]
For any differentiable random
variable $F$ and any $r=(r_1,...,r_m) \in [0,T]^m$, $D^m F(r)$ is an
element of $\mathcal{H}^{\otimes m}$ which will be denoted by
$D^m_r F$.


A random variable $F$ is said to be {\emph{smooth} if it belongs to $\D^{\infty}$, 
and a smooth random variable $F$ is said to be {\emph{non-degenerate}} if
$\|DF\|_{\hact}^{-1} \in \cap_{p \geq 1} L^p(\Omega)$.
Owing to \cite[Theorem 2.1.4]{Nualart-llibre}, we know that a non-degenerate random variable has a $\mathcal{C}^{\infty}$
density.

For any $t\in [0,T]$, let $\mathcal{F}_t$ be the $\sigma$-field generated
by the random variables $\{W_s(h), h \in \mathcal{H}, 0 \leq s
\leq t\}$ and the $\P$-null sets, where $W_t(h):=W(1_{[0,t]}h)$.

\section{Spatially homogeneous SPDEs}
\label{sec:spde}

We gather here a general result on existence and uniqueness of mild solution for our SPDE (\ref{eq:21}) and 
the main results on Malliavin calculus applied to it, namely Malliavin differentiability and existence and smoothness of density.
As usual, we will also focus on the main examples of application that we have in mind, which are the stochastic heat and wave equations with 
$d\geq 1$ and $d\in \{1,2,3\}$, respectively.

We recall that, by definition, a mild solution of (\ref{eq:21}) is an $\tf_t$-adapted random field $\{u(t,x),\, (t,x)\in [0,T]\times \red\}$
such that the following stochastic integral equation is satisfied:
\begin{align}
u(t,x)& = I_0(t,x) + \int_0^t \!\! \int_{\mathbb{R}^{d}} \Gamma(t-s,x-y) \sigma(u(s,y))\, W(ds,dy) \nonumber\\
& \qquad  + \int_0^t \!\! \int_{\mathbb{R}^{d}} b(u(t-s,x-y)) \, \Gamma(s,dy) ds, \qquad \P\text{-a.s.},
\label{eq:22}
\end{align}
for all $(t,x) \in [0,T] \times \red$. Here, $\Gamma$ denotes the fundamental solution
associated to $\cl$ and $I_0(t,x)$ is the contribution of the initial conditions, which we define below. 

The (real-valued) stochastic integral on the right-hand side of (\ref{eq:22}) is understood with 
respect to the {\emph{cylindrical Wiener process}} that can be naturally associated to our spatially homogeneous noise $W$
(see \cite{Nualart-QuerPOTA,Dalang-Quer} and also \cite{Walsh-LNM,Dalang-EJP1999}). In particular, we will assume that Hypothesis \ref{hyp:mu} is satisfied.
Concerning the last integral on the
right-hand side of (\ref{eq:22}), we point out that we use the notation ``$\Gam(s,dy)$'' because we
will assume that $\Gam(s)$ is a measure on $\red$.

As far as the term $I_0(t,x)$ is concerned, if $\cl$ is a parabolic-type operator
and we consider the initial condition (\ref{2.1.1}), then
\beq 
I_0(t,x)= \left( \Gam(t) * u_0\right)(x)=\int_{\red} u_0(x-y)\, \Gam(t,dy). 
\label{2.3} 
\eeq 
On the other hand, in
the case where $\cl$ is second order in time with initial values (\ref{2.1.2}),
\beq
I_0(t,x) = \left( \Gam(t) * v_0\right)(x) + \frac{\partial}{\partial t} \left( \Gam(t) * u_0\right)(x).
\label{2.4}
\eeq

\begin{example}
Owing to the considerations in \cite[Section 3]{Dalang-EJP1999} (see also \cite[Examples 4.2 and 4.3]{Nualart-QuerPOTA}),
in the case of the stochastic heat equation in any space dimension $d\geq 1$ and the stochastic wave equation in
dimensions $d=1,2,3$, the fundamental solutions are well-known and the conditions in Hypothesis \ref{hyp:mu} are 
satisfied if and only if 
\beq
\int_{\red} \frac{1}{1+|\xi|^2}\, \mu(d\xi) < + \infty.
\label{eq:18}
\eeq
\end{example}

\bigskip

We shall consider the following assumption on the initial conditions. In the case of the stochastic heat equation in any space dimension 
(resp. wave equation with dimension $d=1,2,3$), sufficient conditions on $u_0$ (resp. $u_0$, $v_0$) implying that the hypothesis below
is fulfilled are provided in \cite[Lemma 4.2]{Dalang-Quer}.

\begin{hyp}\label{hyp:ic}
 $(t,x) \mapsto I_0(t,x)$ is continuous and 
$\sup_{(t,x)\in[0,T]\times \red} |I_0(t,x)|<+\infty$.
\end{hyp}
  
\medskip

The following well-posedness result, which is a quotation of \cite[Thm. 4.3]{Dalang-Quer}, is a slight extension of the results 
in \cite{Dalang-EJP1999}.

\begin{thm}\label{existence}
Assume that Hypotheses \ref{hyp:mu} and \ref{hyp:ic} are satisfied and that $\sig$ and $b$ are Lipschitz functions. 
Then there exists a unique solution $\{u(t,x),\, (t,x)\in [0,T]\times \red\}$ of equation (\ref{eq:22}). 
Moreover, for all $p\geq 1$,
\[
\sup_{(t,x)\in [0,T]\times \red} E(|u(t,x)|^p) <+\infty.
\]
\end{thm}

\bigskip

Let us now deal with the Malliavin differentiability of the solution $u(t,x)$ of (\ref{eq:22}). 
For this, we consider the Gaussian context described in Section \ref{sec:malliavin}. 
The following proposition summarizes a series of results in \cite{Marquez-Mellouk-SarraSPA2001,Quer-Sanz-Bernoulli,Nualart-QuerPOTA}.
For the statement, we will use the following notation: for any $m\in \mathbb{N}$, set $\bar s:=(s_1,\dots,s_m)\in [0,T]^m$, $\bar z:=(z_1,\dots,z_m)\in (\red)^m$,
$\bar s(i):=(s_1,\dots,s_{i-1},s_{i+1},\dots,s_m)$ (resp. $\bar z(i)$), and, for any function $f$ and variable $X$ for which it makes sense, set
\[
\Del^m(f,X):= D^mf(X)-f'(X)D^mX.
\]
Note that $\Del^m(f,X)=0$ for $m=1$ and, if $m>1$, it only involves iterated Malliavin derivatives up to order $m-1$.

\begin{prop}\label{prop:mal-dif}
Assume that Hypothesis \ref{hyp:mu} is satisfied and, for some $m\in \mathbb{N}\cup \{\infty\}$, $\sigma, b \in \mathcal{C}^m(\R)$ 
and their derivatives of order greater than or equal to one are bounded,
Then, for all $(t,x)\in [0,T]\times \red$, the random variable $u(t,x)$ belongs to $\D^{j,p}$ for any $j=1,\dots,m$ and $p\geq 1$. 
Furthermore, for any $j\in \{1,\dots,m\}$ and $p\geq 1$, the iterated Malliavin derivative $D^ju(t,x)$ satisfies the following 
equation in $L^p(\Om;\hact^{\otimes j})$:
\begin{align} 
& D^ju(t,x) = Z^j(t,x)  \nonumber \\
& \qquad + \int_0^t\int_{\red} \Gam(t-s,x-y) [\Del^j(\sig,u(s,y))+D^j u(s,y) \sig'(u(s,y))] W(ds,dy)  \nonumber \\
& \qquad + \int_0^t\int_{\red} [\Del^j(b,u(t-s,x-y))+D^j u(t-s,x-y) b'(u(t-s,x-y))]\, \Gam(s,dy)ds,
\label{eq:113}
\end{align}
where $Z^j(t,x)$ is the element of $L^p(\Om;\hact^{\otimes j})$ given by
\[
Z^j(t,x)_{\bar s,\bar z}= \sum_{i=1}^j \Gam(t-s_i,x-dz_i) D^{j-1}_{\bar s(i), \bar z(i)} \sig(u(s_i,z_i)).
\]
\end{prop}

A detailed description of the construction of Hilbert-space-valued stochastic integrals as the one in equation (\ref{eq:113}) can be found in 
\cite[Sec. 3]{Nualart-QuerPOTA}. Indeed, as proved in \cite[Sec. 3.6]{Dalang-Quer}, these kind of integrals turn out to be equivalent to 
Hilbert-space-valued stochastic integrals {\emph{\`a la}} Da Prato and Zabczyk \cite{DaPrato-Zabczyk}. 

The above Proposition \ref{prop:mal-dif} can be used to obtain the following results on existence and smoothness of the density 
for the solution $u(t,x)$. They are direct consequences of Theorems 5.2 and 6.2 in \cite{Nualart-QuerPOTA}, with the only difference that 
the latter consider vanishing initial conditions.

\begin{thm}
Assume that Hypotheses \ref{hyp:mu} and \ref{hyp:ic} are satisfied, $b, \sig \in \mathcal{C}^1$ have a bounded derivative, 
and $|\sig(z)|\geq c>0$ for all $z\in \re$. Then, for all $(t,x)\in (0,T]\times \red$, the random variable $u(t,x)$ has a law which 
is absolutely continuous with respect to the Lebesgue measure.
\end{thm}

\begin{thm}
Assume that Hypotheses \ref{hyp:mu} and \ref{hyp:ic} are satisfied, $\sigma, b \in \mathcal{C}^\infty$ 
and their derivatives of order greater than or equal to one are bounded, and that $|\sig(z)|\geq c>0$ for all $z\in \re$.
Moreover, suppose that, for some $\gam>0$, 
\beq
 C\, t^\gam \leq \int_0^t \!\! \int_{\red} |\tf \Gam(s)(\xi)|^2 \, \mu(d\xi)ds, \qquad t\in (0,1).
\label{eq:999}
\eeq
Then, for every $(t,x)\in (0,T]\times \red$, the law of the random variable $u(t,x)$ has a $\cinf$ density.
\end{thm}

As commented in the Introduction, both results apply to the stochastic heat equation with $d\geq 1$ and the stochastic wave equation with $d\in \{1,2,3\}$ 
provided that (\ref{eq:18}) is satisfied, since condition (\ref{eq:999}) holds for these examples with $\gam=1$ and $\gam=3$, respectively.

\section{Auxiliary results}
\label{sec:aux}

This section is devoted to prove the main two ingredients needed in the proof of Theorem \ref{thm:main}.
The first one establishes a suitable uniform bound for the norm of the iterated Malliavin derivative of the 
solution $u(t,x)$ in small time intervals. The second one deals with the negative moments of the corresponding 
Malliavin matrix, which here simply reduces to the norm of the Malliavin derivative of $u(t,x)$.

\begin{lemma} \label{lemma:0}
Let $0\leq a<e\leq T$ and $p\geq 1$. Assume that Hypotheses \ref{hyp:mu} and \ref{hyp:ic} are satisfied and that, for some $m\in \N$, the coefficients 
$b,\sig$ belong to $\mathcal{C}^m$ and all their derivatives of order greater than or equal to one are bounded. 
Then, there exists a positive constant $C$, which is independent of $a$ and $e$, such that, for all $\del\in (0,e-a]$:
\beq
\sup_{(\tau,y) \in [e-\del,e]\times \red} \E \left( \Vert D^j u(\tau,y) \Vert^{2p}_{\mathcal{H}^{\otimes j}_{e-\del,e}} \right) \leq C
\, \Phi(\del)^{jp},
\label{eq:46}
\eeq
for all $j\in \{1,\dots,m\}$, where we remind that, for all $t\geq 0$,
\[
\Phi(t)= \int_0^t \!\! \int_{\red} |\tf \Gam(s)(\xi)|^2\, \mu(d\xi) ds.
\]
\end{lemma}

\begin{proof}
It is similar to that of \cite[Lem. 3.4]{Eulalia-Quer-SPA}, where a conditioned version of this result for the stochastic heat equation 
has been proved. Precisely, as already pointed out in \cite[Rmk. 3.5]{Eulalia-Quer-SPA}, in our general setting we need to smooth the fundamental solution 
$\Gam$ as follows. Let $\psi\in \mathcal{C}_0^{\infty}(\mathbb{R}^d)$ be such that $\psi\geq 0$,
its support is contained in the unit ball of $\mathbb{R}^d$ and $\int_{\mathbb{R}^d} \psi(x)dx=1$.
For $n\in \N$, set $\psi_n(x):=n^d \psi(nx)$ and, for all $t$, $\Gamma_n(t):=\psi_n \ast \Gamma(t)$. 
It is well-known that $\Gamma_n(t)$ belongs to $\mathcal{S}(\mathbb{R}^d)$.

Let us now consider $\{u_n(t,x),\, (t,x)\in [0,T]\times \red\}$ the unique solution of 
\begin{align*}
u_n(t,x)& = I_0(t,x) + \int_0^t \!\! \int_{\mathbb{R}^{d}} \Gamma_n (t-s,x-y) \sigma(u_n(s,y))\, W(ds,dy) \\
& \qquad  + \int_0^t \!\! \int_{\mathbb{R}^{d}} b(u_n(t-s,x-y)) \, \Gamma(s,dy) ds. 
\end{align*}
Since $\Gam_n(t)$ is a smooth function (such as in the case of the heat equation), we can mimic the proof of 
\cite[Lem. 3.4]{Eulalia-Quer-SPA}, so that we end up with estimate (\ref{eq:46}) with $u$ replaced by $u_n$. 
Indeed, we should remark here that the term involving the pathwise integral with respect to     
$\Gamma(s,dy) ds$ does not cause any problem since we only need to use that $\Gam(t,\red)$ is 
uniformly bounded in $t$, which is part of Hypothesis \ref{hyp:mu}.

On the other hand, a direct consequence of the proofs of \cite[Thm. 1]{Quer-Sanz-Bernoulli} and \cite[Prop. 6.1]{Nualart-QuerPOTA} is 
that, for all $(t,x)\in [e-\del,e]\times \red$ and $j\in \{1,\dots,m\}$, 
\[
 D^j u(t,x) =L^2(\Om; \hac^{\otimes j}_{e-\del,e})-\lim_{n\rightarrow \infty} D^j u_n(t,x).
\]
Therefore, writing down the corresponding convergence of norms and taking supremum over $[e-\del,e]\times\red$, we conclude the proof.
\end{proof}

\begin{prop}\label{prop:matrix}
Assume that Hypotheses \ref{hyp:mu} and \ref{hyp:ic} are satisfied, that $b, \sigma$ are $\mathcal{C}^1$ functions with bounded derivatives and 
that $|\sig(z)|\geq c >0$ for all $z\in \re$. Moreover, suppose that, for some $\gam>0$, 
\beq
 C\, t^\gam \leq \int_0^t \!\! \int_{\red} |\tf \Gam(s)(\xi)|^2 \, \mu(d\xi)ds, \qquad t\in (0,1).
\label{eq:9999}
\eeq
Then, for any $p>0$, 
there exists a constant $C>0$ such that, for all $(t,x)\in (0,T]\times \red$,
\[
 \E \left( \|D u(t,x)\|^{-2p}_{\hact}\right) \leq C\, \Phi(t)^{-p}.
\]
\end{prop}

\begin{proof}
The proof's structure is analogous as that of the proofs of \cite[Thm. 6.2]{Nualart-QuerPOTA} and \cite[Prop. 4.3]{Eulalia-Quer-SPA}, 
so we will only sketch the main steps.

First, owing to \cite[Lem. 2.3.1]{Nualart-llibre}, it suffices to check that, for any $q>2$, 
there exists $\ep_0=\ep_0(q)>0$ such that, for all $\ep \leq \ep_0$,
\beq
\P \left\{ \Phi(t)^{-1} \| D u(t,x) \|^2_{\hact} < \ep \right\} \leq C \ep^q. 
\label{eq:449}
\eeq
Note that the Malliavin derivative $Du(t,x)$ verifies the following equation in $\hact$ (take $m=1$ in (\ref{eq:113})):
\begin{align*}
D u(t,x) & = \sigma(u(\cdot,\star))\Gamma(t-\cdot,x-\star) \\
& \qquad + \int_0^t \!\! \int_{\mathbb{R}^{d}} \Gamma(t-s,x-y) \sigma'(u(s,y)) D u(s,y)W(ds,dy) \\
& \qquad + \int_0^t \!\! \int_{\mathbb{R}^{d}}  b'(u(s,x-y)) D u(s,x-y)\Gamma(t-s,dy) ds.
\end{align*}
Then, for any small $\del>0$ one proves that
\begin{align}
\P \left\{ \Phi(t)^{-1} \| D u(t,x) \|^2_{\hact} < \ep \right\} & \leq 
\P \left\{ \Phi(t)^{-1} I(t,x;\del) \geq c\, \Phi(t)^{-1} \Phi(\del)- \ep  \right\} \nonumber \\
& \leq \left( c \, \Phi(t)^{-1} \Phi(\del)- \ep\right) ^{-p}  \Phi(t)^{-p} \, \E (|I(t,x;\del)|^p),  
\label{eq:47}
\end{align}
where $I(t,x;\del):= \|R_1(t,x;\del)\|^2_{\hac_{t-\del,t}} + \|R_2(t,x;\del)\|^2_{\hac_{t-\del,t}}$ and 
\[
 R_1(t,x;\del):=\int_\cdot^t \!\!\int_{\red} \Gam(t-s,x-y)\sig'(u(s,y)) D u(s,y) W(ds,dy),
\]
\[
 R_2(t,x;\del):=\int_\cdot^t \!\!\int_{\red} b'(u(t-s,x-y)) D u(t-s,x-y) \, \Gam(s,dy)ds.
\]
Using the above Lemma \ref{lemma:0} and applying standard integral estimates, one checks that
\[
 \E (|I(t,x;\del)|^p) \leq C \, \Phi(\del)^p \left( \Phi(\del)^p + \Psi(\del)^p\right),
\]
where we have set 
\[
\Psi(s):= \int_0^s \Gam(r,\red)dr. 
\]
Thus, going back to (\ref{eq:47}), we obtain 
\[
 \P \left\{ \Phi(t)^{-1} \| D u(t,x) \|^2_{\hact} < \ep \right\} 
\leq C \left( c \, \Phi(t)^{-1} \Phi(\del)- \ep\right) ^{-p}  \Phi(t)^{-p} \Phi(\del)^p \left( \Phi(\del)^p + \Psi(\del)^p\right).
\]
At this point, taking a small enough $\ep_0$ if necessary, we can choose $\del=\del(\ep)$ such that 
\beq
 \frac{c}{2} \, \Phi(t)^{-1} \Phi(\del) = \ep.
\label{eq:48}
\eeq
Hence, we have
\[
 \P \left\{ \Phi(t)^{-1} \| D u(t,x) \|^2_{\hact} < \ep \right\} 
\leq C \left( \Phi(\del)^p + \Psi(\del)^p\right).
\]
Note, on the one hand, that condition (\ref{eq:48}) implies $\Phi(\del)\leq C \Phi(T) \ep \leq C \ep$.
On the other hand, by Hypothesis \ref{hyp:mu} we have $\Psi(\del)\leq C \del$. Hence, the assumption (\ref{eq:9999}) and what we have just said above 
let us infer that $\Psi(\del)\leq C \ep^{\frac1\gam}$. Therefore,
\[
 \P \left\{ \Phi(t)^{-1} \| D u(t,x) \|^2_{\hact} < \ep \right\} 
\leq C \left(\ep^p + \ep^{\frac{p}{\gam}}\right),
\]
so taking $p=q(\gamma \vee 1)$ we conclude that (\ref{eq:449}) is satisfied.
\end{proof}

\section{Proof of the main result}
\label{sec:upper-bound}

In this section, we are going to prove Theorem \ref{thm:main}. On the one hand, we note that Proposition \ref{prop:mal-dif} implies 
that, for any $(t,x)\in (0,T]\times \red$, the random variable $u(t,x)$ belongs to $\mathbb{D}^{2,p}$ for all $p\geq 1$. 
Moreover, an immediate consequence of Proposition \ref{prop:matrix} is that the Malliavin matrix associated to $u(t,x)$ has 
negative moments of all orders. Thus, applying a general criterion of the Malliavin calculus (see e.g. \cite[Prop. 2.1.5]{Nualart-llibre} or
\cite[Thm. 4.1]{Malliavin}), we obtain that the law of $u(t,x)$ has a density and it is a continuous function.   

On the other hand, as explained in the Introduction, the proof of (\ref{eq:gb}) is a matter 
of following exactly the same arguments as in \cite[Sec. 5]{Eulalia-Quer-SPA} and invoking the results of the previous section.
Let us sketch the main steps to follow.

To start with, we use the formula for the density arising from the application of the integration-by-parts formula in the Malliavin calculus context
(see e.g. \cite[Prop. 2.1.1]{Nualart-llibre}). Precisely, denoting the density of $u(t,x)$ by $p_{t,x}$, we have
\[
p_{t,x}(y)=E \left( {\bf{1}}_{\{ u(t,x)>y \}} \del \left( \frac{Du(t,x)}{\|Du(t,x)\|^2_{\hact}}\right)\right), \quad y\in \re,
\]
where here $\del$ denotes the divergence operator or Skorohod integral, that is the adjoint of the Malliavin derivative operator (see \cite[Ch. 1]{Nualart-llibre}).
  
Next, taking into account the equation satisfied by $u(t,x)$ (i.e. (\ref{eq:22})) and applying \cite[Prop. 2.1.2]{Nualart-llibre}, 
we obtain
\begin{align}
p_{t,x}(y) & \leq C\;  \P \left\{|M_t| > |y-I_0(t,x)|-c_3T \right\}^{\frac1q} \nonumber \\
& \qquad \quad \times \left\{ \E \left( \Vert D u(t,x) \Vert_{\mathcal{H}_T}^{-1}\right) +  
\left( \E \Vert D^2 u(t,x)\Vert_{\mathcal{H}_T^{\otimes 2}}^\alpha \right)^{\frac1\al}  
\left( \E \Vert D u(t,x) \Vert^{-2\beta}_{\mathcal{H}_T} \right)^{\frac1\beta} \right\},
\label{eq:223}
\end{align}
where $\alpha, \beta, q$ are any positive real numbers satisfying $\frac{1}{\alpha}+\frac{1}{\beta}+\frac{1}{q}=1$.
In the above expression, $M_t$ denotes the {\emph{martingale part}} of the solution $u(t,x)$, that is 
\[
M_t=\int_0^t \!\! \int_{\red} \Gam(t-s,x-y) \sigma(u(s,y)) W(ds,dy), 
\]
and the term $c_3 T$ comes from the fact that, due to Hypothesis \ref{hyp:mu} and the boundedness of $b$, for all $(t,x)\in (0,T]\times \red$,
\[
 \left| \int_0^t \!\! \int_{\red} \Gam(t-s,x-y) b(u(s,y)) \, dy ds \right| \leq c_3 \, T, \qquad \P\text{-a.s.}
\]

In order to estimate the terms in (\ref{eq:223}), we first apply the exponential martingale inequality in order to get 
a suitable exponential bound of the probability in (\ref{eq:223}) (using that $\langle M \rangle_t \leq C\, \Phi(t)$), and then we conveniently apply 
Lemma \ref{lemma:0} and Proposition \ref{prop:matrix}. Thus
\[
p_{t,x}(y) \leq c_1\, \Phi(t)^{-1/2} \exp \biggl( -\frac{(|y-I_0(t,x) \vert-c_3 T)^2}{c_2 \Phi(t)}\biggr), \quad y\in \re,
\]
where the constants $c_1, c_2, c_3$ do not depend on $(t,x)$, so we conclude the proof of Theorem \ref{thm:main}. 
\qed



\begin{thebibliography}{10}

\bibitem{Bally-PardouxPOTA1998}
V.~Bally and E.~Pardoux, \emph{Malliavin calculus for white noise driven
  parabolic {SPDE}s}, Potential Anal. \textbf{9} (1998), no.~1, 27--64.

\bibitem{Carmona-NualartPTRF1988}
R.~Carmona and D.~Nualart, \emph{Random nonlinear wave equations: smoothness of
  the solutions}, Probab. Theory Related Fields \textbf{79} (1988), no.~4,
  469--508.

\bibitem{Carmona-Molchanov}
R.A. Carmona and S.~A. Molchanov, \emph{Parabolic {A}nderson problem and
  intermittency}, Mem. Amer. Math. Soc. \textbf{108} (1994), no.~518, viii+125.
  \MR{1185878 (94h:35080)}

\bibitem{Conus-Dalang}
D.~Conus and R.C. Dalang, \emph{The non-linear stochastic wave equation in high
  dimensions}, Electron. J. Probab. \textbf{13} (2008), no. 22, 629--670.
  \MR{2399293 (2009c:60170)}

\bibitem{DaPrato-Zabczyk}
G.~Da~Prato and J.~Zabczyk, \emph{Stochastic equations in infinite dimensions},
  Encyclopedia of Mathematics and its Applications, vol.~44, Cambridge
  University Press, Cambridge, 1992.

\bibitem{Dalang-Quer}
R.~Dalang and L.~Quer-Sardanyons, \emph{Stochastic integrals for spde's: A
  comparison}, Expo. Math. \textbf{29} (2011), 67--109.

\bibitem{Dalang-EJP1999}
R.C. Dalang, \emph{Extending the martingale measure stochastic integral with
  applications to spatially homogeneous s.p.d.e.'s}, Electron. J. Probab.
  \textbf{4} (1999), no.\ 6, 29 pp.\ (electronic).

\bibitem{Dalang-FrangosAP1998}
R.C. Dalang and N.E. Frangos, \emph{The stochastic wave equation in two spatial
  dimensions}, Ann. Probab. \textbf{26} (1998), no.~1, 187--212.

\bibitem{Dalang-Khosh-Nualart-PTRF2009}
R.C. Dalang, D.~Khoshnevisan, and E.~Nualart, \emph{Hitting probabilities for
  systems for non-linear stochastic heat equations with multiplicative noise},
  Probab. Theory Related Fields \textbf{144} (2009), no.~3-4, 371--427.
  \MR{2496438 (2010g:60151)}

\bibitem{Dalang-MuellerEJP2003}
R.C. Dalang and C.~Mueller, \emph{Some non-linear {S}.{P}.{D}.{E}.'s that are
  second order in time}, Electron. J. Probab. \textbf{8} (2003), no. 1, 21 pp.
  (electronic).

\bibitem{Fournier}
N.~Fournier and J.~Printems, \emph{Absolute continuity for some one-dimensional
  processes}, Bernoulli \textbf{16} (2010), no.~2, 343--360.

\bibitem{Guerin-Meleard-Eulalia}
H.~Gu{\'e}rin, S.~M{\'e}l{\'e}ard, and E.~Nualart, \emph{Estimates for the
  density of a nonlinear {L}andau process}, J. Funct. Anal. \textbf{238}
  (2006), no.~2, 649--677. \MR{2253737 (2008e:60164)}

\bibitem{Hu-Nualart-Song}
Y.~Hu, D.~Nualart, and J.~Song, \emph{A nonlinear stochastic heat equation:
  H\"older continuity and smoothness of the density of the solution}, ArXiv
  Preprint arXiv:1110.4855v1.

\bibitem{Hu-Nualart-Song-AP2011}
\bysame, \emph{Feynman-{K}ac formula for heat equation driven by fractional
  white noise}, Ann. Probab. \textbf{39} (2011), no.~1, 291--326. \MR{2778803
  (2012b:60208)}

\bibitem{Karczewska-Zabczyk2000}
A.~Karczewska and J.~Zabczyk, \emph{Stochastic {PDE}'s with function-valued
  solutions}, Infinite dimensional stochastic analysis (Amsterdam, 1999), Verh.
  Afd. Natuurkd. 1. Reeks. K. Ned. Akad. Wet., vol.~52, R. Neth. Acad. Arts
  Sci., Amsterdam, 2000, pp.~197--216. \MR{2002h:60132}

\bibitem{Kohatsu-PTRF2003}
A.~Kohatsu-Higa, \emph{Lower bounds for densities of uniformly elliptic random
  variables on {W}iener space}, Probab. Theory Related Fields \textbf{126}
  (2003), no.~3, 421--457. \MR{1992500 (2004d:60141)}

\bibitem{Malliavin}
P.~Malliavin, \emph{Stochastic analysis}, Grundlehren der Mathematischen
  Wissenschaften [Fundamental Principles of Mathematical Sciences], vol. 313,
  Springer-Verlag, Berlin, 1997. \MR{1450093 (99b:60073)}

\bibitem{Marinelli-Nualart-Quer}
C.~Marinelli, E.~Nualart, and L.~Quer-Sardanyons, \emph{Existence and
  regularity of the density for the solution to semilinear dissipative
  parabolic spdes}, arXiv:1202.4610.

\bibitem{Marquez-Mellouk-SarraSPA2001}
Mellouk~M. M{\'a}rquez-Carreras, D. and M.~Sarr{\`a}, \emph{On stochastic
  partial differential equations with spatially correlated noise: smoothness of
  the law}, Stochastic Process. Appl. \textbf{93} (2001), no.~2, 269--284.

\bibitem{Millet-Sanz-AP1999}
A.~Millet and M.~Sanz-Sol{\'e}, \emph{A stochastic wave equation in two space
  dimension: smoothness of the law}, Ann. Probab. \textbf{27} (1999), no.~2,
  803--844.

\bibitem{Nourdin-Viens}
I.~Nourdin and F.G. Viens, \emph{Density formula and concentration inequalities
  with {M}alliavin calculus}, Electron. J. Probab. \textbf{14} (2009), no. 78,
  2287--2309. \MR{2556018 (2011a:60147)}

\bibitem{Nualart-llibre}
D.~Nualart, \emph{The {M}alliavin calculus and related topics}, second ed.,
  Probability and its Applications (New York), Springer-Verlag, Berlin, 2006.

\bibitem{Nualart-QuerPOTA}
D.~Nualart and L.~Quer-Sardanyons, \emph{Existence and smoothness of the
  density for spatially homogeneous {SPDE}s}, Potential Anal. \textbf{27}
  (2007), no.~3, 281--299.

\bibitem{Nualart-Quer-IDAQP}
\bysame, \emph{Gaussian density estimates for solutions to quasi-linear
  stochastic partial differential equations}, Stochastic Process. Appl.
  \textbf{119} (2009), no.~11, 3914--3938. \MR{2552310 (2011g:60113)}

\bibitem{Nualart-Quer-SPA}
\bysame, \emph{Optimal {G}aussian density estimates for a class of stochastic
  equations with additive noise}, Infin. Dimens. Anal. Quantum Probab. Relat.
  Top. \textbf{14} (2011), no.~1, 25--34. \MR{2785746 (2012e:60155)}

\bibitem{Eulalia-Quer-SPA}
E.~Nualart and L.~Quer-Sardanyons, \emph{Gaussian estimates for the density of
  the non-linear stochastic heat equation in any space dimension}, Stochastic
  Process. Appl. \textbf{122} (2012), no.~1, 418--447. \MR{2860455}

\bibitem{Pardoux-Zhang-JFA93}
{\'E}.~Pardoux and T.S. Zhang, \emph{Absolute continuity of the law of the
  solution of a parabolic {SPDE}}, J. Funct. Anal. \textbf{112} (1993), no.~2,
  447--458. \MR{MR1213146 (94k:60095)}

\bibitem{Peszat-JEE2002}
S.~Peszat, \emph{The {C}auchy problem for a nonlinear stochastic wave equation
  in any dimension}, J. Evol. Equ. \textbf{2} (2002), no.~3, 383--394.
  \MR{2003k:60157}

\bibitem{Peszat-Zabczyk-SPA1997}
S.~Peszat and J.~Zabczyk, \emph{Stochastic evolution equations with a spatially
  homogeneous {W}iener process}, Stochastic Process. Appl. \textbf{72} (1997),
  no.~2, 187--204. \MR{MR1486552 (99k:60166)}

\bibitem{Peszat-Zabczyk-PTRF2000}
\bysame, \emph{Nonlinear stochastic wave and heat equations}, Probab. Theory
  Related Fields \textbf{116} (2000), no.~3, 421--443. \MR{2001f:60071}

\bibitem{Quer-Sanz-JFA}
L.~Quer-Sardanyons and M.~Sanz-Sol{\'e}, \emph{Absolute continuity of the law
  of the solution to the 3-dimensional stochastic wave equation}, J. Funct.
  Anal. \textbf{206} (2004), no.~1, 1--32. \MR{2 024 344}

\bibitem{Quer-Sanz-Bernoulli}
\bysame, \emph{A stochastic wave equation in dimension 3: smoothness of the
  law}, Bernoulli \textbf{10} (2004), no.~1, 165--186. \MR{2 044 597}

\bibitem{Sanz-book}
M.~Sanz-Sol{\'e}, \emph{Malliavin calculus}, Fundamental Sciences, EPFL Press,
  Lausanne, 2005, With applications to stochastic partial differential
  equations. \MR{2167213 (2006h:60005)}

\bibitem{Schwartz}
L.~Schwartz, \emph{Th\'eorie des distributions}, Publications de l'Institut de
  Math\'ematique de l'Universit\'e de Strasbourg, No. IX-X. Nouvelle \'edition,
  enti\'erement corrig\'ee, refondue et augment\'ee, Hermann, Paris, 1966.
  \MR{35 \#730}

\bibitem{Walsh-LNM}
J.B. Walsh, \emph{An introduction to stochastic partial differential
  equations}, \'Ecole d'\'et\'e de probabilit\'es de Saint-Flour, XIV---1984,
  Lecture Notes in Math., vol. 1180, Springer, Berlin, 1986, pp.~265--439.
  \MR{88a:60114}

\end{thebibliography}

\def\cprime{$'$} \def\cprime{$'$} \def\cprime{$'$} \def\cprime{$'$}
  \def\polhk#1{\setbox0=\hbox{#1}{\ooalign{\hidewidth
  \lower1.5ex\hbox{`}\hidewidth\crcr\unhbox0}}}
\providecommand{\bysame}{\leavevmode\hbox to3em{\hrulefill}\thinspace}
\providecommand{\MR}{\relax\ifhmode\unskip\space\fi MR }
\providecommand{\MRhref}[2]{%
  \href{http://www.ams.org/mathscinet-getitem?mr=#1}{#2}
}
\providecommand{\href}[2]{#2}

\end{document}